\documentclass[a4paper,10pt]{amsart}

\usepackage[english]{babel}
\usepackage[utf8]{inputenc}

\usepackage{amssymb}
\usepackage{amsmath}
\usepackage{amsthm}
\usepackage{bbm}

\usepackage{enumerate}
\usepackage{tikz}
\usepackage{marginnote}

\usepackage[arrow, matrix, curve]{xy}

\newcommand{\bbN}{\mathbb{N}}

\newcommand{\bbR}{\mathbb{R}}

\newcommand{\calB}{\mathcal{B}}

\newcommand{\calM}{\mathcal{M}}
\newcommand{\calP}{\mathcal{P}}

\newcommand{\id}{I}

\DeclareMathOperator{\linSpan}{span}

\newcommand{\modulus}[1]{\left\vert #1 \right\rvert}
\DeclareMathOperator{\Bands}{Bands}
\DeclareMathOperator{\ProjBands}{PrBands}
\DeclareMathOperator{\BandProj}{BandPr}

\theoremstyle{definition}
\newtheorem{definition}{Definition}[section]

\newtheorem{remarks}[definition]{Remarks}
\newtheorem{example}[definition]{Example}
\newtheorem{examples}[definition]{Examples}

\theoremstyle{plain}
\newtheorem{proposition}[definition]{Proposition}
\newtheorem{lemma}[definition]{Lemma}
\newtheorem{theorem}[definition]{Theorem}
\newtheorem{corollary}[definition]{Corollary}

\begin{document}

\title[Disjointness, band and projections]{On disjointness, bands and projections in partially ordered vector spaces}
\author{Jochen Gl\"uck}
\address{Jochen Gl\"uck, Fakultät für Informatik und Mathematik, Universität Passau, Innstraße 33, D-94032 Passau, Germany}
\email{jochen.glueck@uni-passau.de}
\date{\today}



\begin{abstract}
	Disjointness, bands, and band projections are a classical and essential part of the structure theory of vector lattices. If $X$ is such a lattice, those notions seem -- at first glance -- intimately related to the lattice operations on $X$. The last fifteen years, though, have seen an extension of all those concepts to a much larger class of ordered vector spaces.
	
	In fact if $X$ is an Archimedean ordered vector space with generating cone, or a member of the slightly larger class of pre-Riesz spaces,  then the notions of disjointness, bands and band projections can be given proper meaning and give rise to a non-trivial structure theory.
	
	The purpose of this note is twofold: (i) We show that, on any pre-Riesz space, the structure of the space of all band projections is remarkably close to what we have in the case of vector lattices. In particular, this space is a Boolean algebra. (ii) We give several criteria for a pre-Riesz space to already be a vector lattice. These criteria are coined in terms of disjointness and closely related concepts, and they mark how lattice-like the order structure of pre-Riesz spaces can get before the theory collapses to the vector lattice case.
\end{abstract}

\maketitle

\section{Introduction} \label{section:introduction}

\subsection*{Disjointness}

Two elements $x$ and $y$ of a vector lattice $X$ are called \emph{disjoint} if $\modulus{x} \land \modulus{y} = 0$ -- a notion that is well-motivated by the case where $X$ is one of the classical function spaces such as $L^p$. A straightforward generalisation to ordered vector spaces that are not lattices seems to be difficult at first glance, as there is no obvious replacement of the modulus of $x$ and $y$. Van Gaans and Kalauch, though, observed more than a decade ago \cite{Gaans2006} that one can circumvent this obstacle by noting that any two elements $x$ and $y$ of a vector lattice $X$ are disjoint if and only if $\modulus{x+y} = \modulus{x-y}$, and that this is in turn true if and only if the sets $\{x+y,-x-y\}$ and $\{x-y,y-x\}$ have the same set of upper bounds. The latter property clearly allows a generalisation to other ordered vector spaces, which gives rise to the following definition.

Let $(X,X_+)$ be an \emph{ordered vector space}, by which we mean that $X$ is a real vector space and $X_+ \subseteq X$ is a non-empty subset of $X$ which satisfies $X_+ \cap (-X_+) = \{0\}$ and $\alpha X_+ + \beta X_+ \subseteq X_+$ for all scalars $\alpha,\beta \in [0,\infty)$ (we call $X_+$ the \emph{positive cone} in $X_+$). Two elements $x,y \in X$ are called \emph{disjoint} if both sets $\{x+y,-x-y\}$ and $\{x-y,y-x\}$ have the same set of upper bounds in $X$. We use the notation $x \perp y$ to denote that $x$ and $y$ are disjoint. Note that $x \perp y$ if and only if $y \perp x$, and that $x \perp x$ if and only if $x = 0$.

If $x,y \in X$ are both \emph{positive} -- i.e., $x,y \in X_+$ -- then one can prove that $x \perp y$ if and only if the infimum of $x$ and $y$ in $X$ exists and is equal to $0$; see \cite[Proposition~2.1]{Glueck2019}.

\subsection*{Disjoint complements and pre-Riesz spaces}

Let $(X,X_+)$ be an ordered vector space and let $S \subseteq X$. The set
\begin{align*}
	S^\perp := \{x \in X: \; x \perp s \text{ for all } s \in S\}
\end{align*}
is called the \emph{disjoint complement} of $S$. We note that $S_1^\perp \supseteq S_2^\perp$ whenever $S_1$ and $S_2$ are two subsets of $X$ such that $S_1 \subseteq S_2$.

From the theory of vector lattices we would expect $S^\perp$ to always be a vector subspace of $X$ -- but it turns out that one can construct examples of ordered spaces where this is not true (see for instance \cite[Example~4.3]{Gaans2006}). On the other hand though, such counterexamples are somewhat pathological: in fact, one can show that $S^\perp$ is always a vector subspace of $X$ if the cone $X_+$ is \emph{generating} in $X$ (i.e., $X = X_+ - X_+$) and $X$ is \emph{Archimedean} (i.e., $nx \le y$ for all $n \in \bbN := \{1,2, \dots\}$ implies $x \le 0$ whenever $x,y$ are two vectors in $X$).

There is also the slightly more general class of \emph{pre-Riesz spaces} that is relevant in this context: an ordered vector space $(X,X_+)$ is called a pre Riesz space if for every non-empty finite set $A \subseteq X$ and every $x \in X$ the following implication is true: if the set of upper bounds of $x+A$ is contained in the set of upper bounds of $A$, then $x \in X_+$. This concept was introduced by  van Haandel in \cite[Definition~1.1(viii)]{Haandel1993}.

If $(X,X_+)$ is a pre-Riesz space and $S \subseteq X$, then the disjoint complement $S^\perp$ is always a vector subspace of $X$; see \cite[Corollary~2.2 and~Section~3]{Gaans2006}. Moreover, we note that every pre-Riesz space has generating cone and that, conversely, every ordered vector space which has generating cone and is, in addition, Archimedean, is a pre-Riesz space \cite[Theorem~3.3]{Gaans2006}. 

The theory of pre-Riesz spaces has undergone a considerable development over the last 15 years. References to papers that deal with bands and projection bands on pre-Riesz spaces are given at the beginning of Sections~\ref{section:bands} and~\ref{section:band-projections}. Further contributions to the theory of pre-Riesz spaces can be found in \cite{Gaans2010, Kalauch2014, Kalauch2014a, Kalauch2019a} and, with a focus on operator theory, in \cite{Kalauch2018, Kalauch2019c, Kalauch2019d, KalauchPreprint5}. The present state of the art in the theory of pre-Riesz spaces is presented in the recent monograph \cite{Kalauch2019}.

\subsection*{Organisation of the paper} In the rest of the introduction we recall a bit more terminology and a simple result about disjointness. In Section~\ref{section:bands} we recall how a band is defined in a pre-Riesz space, and we show a few elementary results about the structure of the set of all bands. In Section~\ref{section:band-projections} we discuss projection bands and band projections. We show, among other things, that the band projections on a pre-Riesz space constitute a Boolean algebra and that, under appropriate assumptions on the space, the intersection of arbitrarily many projection bands is again a projection band. In the final Section~\ref{section:characterisations-of-vector-lattices} we give various sufficient conditions for a pre-Riesz space to be a vector lattice; these conditions are related to several variations of the notion \emph{disjointness}.

\subsection*{Setting the stage}

Throughout the rest of the paper, let $(X,X_+)$ be a pre-Riesz space. 

By an \emph{operator} on $X$ we always mean a linear map $X \to X$, and by a \emph{projection} on $X$ we always mean a linear projection $X \to X$.

We use standard terminology and notation from the theory of ordered vector spaces (which has, to some extent, already been employed above). In particular, we write $x \le y$ (or $y \ge x$) for $x,y \in X$ if $y-x \in X_+$ and we note that the relation $\le$ is a partial order on $X$ which is compatible with the vector space structure. Elements of $X_+$ will be called \emph{positive}. For $x,z \in X$ we denote the \emph{order interval} between $x$ and $z$ by
\begin{align*}
	[x,z] := \{y \in X: \; x \le y \le z\}.
\end{align*}
A linear map $T: X \to X$ is said to be \emph{positive}, which we denote by $T \ge 0$, if $TX_+ \subseteq X_+$, and for two linear maps $S,T: X \to X$ we write $S \le T$ if $T-S \ge 0$.

For each vector subspace $V \subseteq X$ we set $V_+ := V \cap X_+$, and we say that $V$ has \emph{generating cone} or that $V$ is \emph{directed} if $V = V_+ - V_+$. The following simple proposition is quite useful.

\begin{proposition} \label{prop:orthogonal-subspaces-by-positive-parts}
	Let $V,W \subseteq X$ be vector subspaces of $X$ with generating cone. If $V_+ \perp W_+$, then $V \perp W$.
\end{proposition}

Here we use the notation $A \perp B$ for two subsets $A,B \subseteq X$ if $a \perp b$ for each $a \in A$ and each $b \in B$.

\begin{proof}[Proof of Proposition~\ref{prop:orthogonal-subspaces-by-positive-parts}]
	We use that orthogonal complements in pre-Riesz spaces are always vector subspaces. Since $V_+ \subseteq (W_+)^\perp$ we conclude $V \subseteq (W_+)^\perp$. The latter inclusion is equivalent to $W_+ \subseteq V^\perp$, which in turn implies $W \subseteq V^\perp$.
\end{proof}

\section{Bands} \label{section:bands}

This section is in a sense prologue to our main results in Sections~\ref{section:band-projections} and~\ref{section:characterisations-of-vector-lattices}. We briefly recall some basics about bands (Subsection~\ref{subsection:basics-on-bands}), we show that the collection of all bands in $X$ is a complete lattice with respect to set inclusion (Subsection~\ref{subsection:the-lattice-of-all-bands}) and we briefly discuss how the sum of two bands can be computed under certain assumptions (Subsection~\ref{subsection:the-sum-of-two-bands}).

Bands in pre-Riesz spaces were first defined in \cite[Section~5]{Gaans2006} and were further studied in \cite{Gaans2008, Gaans2008a, Kalauch2015, Kalauch2019b}.

\subsection{Basics} \label{subsection:basics-on-bands}

For $S \subseteq X$ we use the notation $S^{\perp\perp} := (S^\perp)^\perp$. Of course, we always have $S \subseteq S^{\perp \perp}$. 

A subset $B \subseteq X$ is called a \emph{band} if $B = B^{\perp \perp}$. For every set $S \subseteq X$ the disjoint complement $S^\perp$ is a band \cite[Proposition~5.5(ii)]{Gaans2006}. As a consequence, for each $S \subseteq X$ the set $S^{\perp \perp}$ is the smallest band in $X$ that contains $S$. Since $(X,X_+)$ is a pre-Riesz space, every band in $X$ is a vector subspace of $X$. Note that if $B$ is a band in $X$ and $0 \le x \le b$ for $x \in X$ and $b \in B$, then we also have $x \in B$.

In the classical theory of vector lattices, the concept of bands is of outstanding importance. For the convenience of the reader we recall a few examples of bands in vector lattices.

\begin{examples} \label{ex:bands-in-function-spaces}
	(a) Let $(\Omega,\mu)$ be a $\sigma$-finite measure space, let $p \in [1,\infty]$ and let $X = L^p(\Omega,\mu)$ with the standard cone. If $A \subseteq \Omega$ is a measurable set, then
	\begin{align*}
		B & = \{f \in X: \; \text{there is a representative of $f$ that vanishes a.\ e.\ on } A\} \\
		  & = \{f \in X: \; \text{every representative of $f$ vanishes a.\ e.\ on } A\}
	\end{align*}
	is a band in $X$, and in fact all bands in $X$ are of this form.
	
	(b) Let $X = C([0,1])$ be the space of continuous real-valued functions on $[0,1]$ and let $0 \le a \le 1$. Then
	\begin{align*}
		B_a = \{f \in X: \; f \text{ vanishes on } [a,1]\}
	\end{align*}
	is a band in $X$ (this example is further discussed in \cite[Example~5 on p.\,63]{Schaefer1974}).
	A general description of bands in spaces of continuous functions over compact sets can, for instance, be found in \cite[Proposition~1.3.13]{Kalauch2019}.
\end{examples}

Interesting examples for bands in a non-lattice ordered pre-Riesz space can for instance be found in the space $\bbR^3$ ordered by the so-called \emph{four ray cone}:

\begin{example} \label{ex:four-ray-cone-bands}
	Let $X = \bbR^3$ and $X_+ := \{\sum_{k=1}^4 \alpha_k v_k: \; \alpha_1,\dots,\alpha_4 \in [0,\infty)\}$, where
	\begin{align*}
		v_1 =
		\begin{pmatrix}
			1 \\ 0 \\ 1
		\end{pmatrix},
		\quad 
		v_2 = 
		\begin{pmatrix}
			0 \\ 1 \\ 1
		\end{pmatrix},
		\quad 
		v_3 =
		\begin{pmatrix}
			-1 \\ 0 \\ 1
		\end{pmatrix},
		\quad 
		v_4 =
		\begin{pmatrix}
			0 \\ -1 \\ 1
		\end{pmatrix}.
	\end{align*}
	The cone $X_+$ is called the \emph{four ray cone} in $\mathbb{R}^3$. All bands in $X$ are computed in \cite[Example~4.4.18]{Kalauch2019}. Besides the two trivial bands $\{0\}$ and $X$ there are four directed bands -- namely the lines spanned by $v_1,\dots,v_4$, respectively. Moreover, there exist two non-directed bands -- namely the lines spanned by
	\begin{align*}
		\begin{pmatrix}
			1 \\ 1 \\ 0
		\end{pmatrix}
		\quad \text{and} \quad 
		\begin{pmatrix}
			1 \\ -1 \\ 0
		\end{pmatrix},
	\end{align*}
	respectively.
\end{example}

\subsection{The lattice of all bands} \label{subsection:the-lattice-of-all-bands}

The following proposition shows that the intersection of any collection of bands in $X$ is again a band.

\begin{proposition} \label{prop:intersections-of-arbitrary-many-bands}
	The intersection of arbitrarily many bands $B_i$ in $X$ (where the indices $i$ are taken from a -- possibly empty -- index set $I$) is again a band, and it is given by
	\begin{align*}
		\bigcap_{i \in I} B_i = \left(\bigcup_{i \in I}B_i^\perp\right)^\perp.
	\end{align*}
\end{proposition}
\begin{proof}
	It suffices to prove the formula. If $x \in \bigcap_{i \in I} B_i$, then $x$ is disjoint to each set $B_i^\perp$, so $x$ is also disjoint to the union $\bigcup_{i \in I} B_i^\perp$. Conversely, fix $i_0 \in I$. Then $B_{i_0}^\perp \subseteq \bigcup_{i \in I} B_i^\perp$ and hence, $B_{i_0} = B_{i_0}^{\perp\perp} \supseteq \left(\bigcup_{i \in I}B_i^\perp\right)^\perp$.
\end{proof}

It is an immediate consequence of this proposition that the set of all bands in $X$ is a complete lattice with respect to set inclusion; let us state this explicitly in the following corollary.

\begin{corollary} \label{cor:lattice-of-bands}
	Let $\Bands(X)$ denote the set of all bands in $X$, ordered by set inclusion. Then every subset $\calB$ of $\Bands(X)$ has a supremum and an infimum in $\Bands(X)$, given by
	\begin{align*}
		\inf \calB = \cap \calB
	\end{align*}
	and
	\begin{align*}
		\sup \calB = \cap \big\{C \in \Bands(X): \,  C \supseteq B \text{ for all } B \in \calB \big\};
	\end{align*}
	in other words, $\Bands(X)$ is a complete lattice.
\end{corollary}

\subsection{The sum of two bands} \label{subsection:the-sum-of-two-bands}

Even in the case of vector lattices, the sum of two bands need not be a band, in general. Let us illustrate this by means of the following simple example.

\begin{example} \label{ex:two-bands-whose-sum-is-not-a-band}
	Let $X = C([-1,1])$ denote the space of all continuous real-valued functions on the interval $[-1,1]$ and endow this space with the standard cone. Then the sets
	\begin{align*}
		B = \{f \in X: \; f|_{[-1,0]} = 0\} \quad \text{and} \quad C = \{f \in X: \; f|_{[0,1]} = 0\} 
	\end{align*}
	are bands in $X$, but their sum $B + B = \{f \in X: \; f(0) = 0\}$ is not a band in $X$.
\end{example}

Another counterexample -- in a non-lattice ordered pre-Riesz space -- can be found in $\bbR^3$ endowed with the four ray cone from Example~\ref{ex:four-ray-cone-bands}. In this space, all non-trivial bands are one-dimensional; hence, the sum of two distinct non-trivial bands cannot be a band in this space.

If, however, the sum of two bands $B$ and $C$ is a band, then we can can compute it be means of the formula $B+C = (B^\perp \cap C^\perp)^\perp$; this is part of the following proposition.

\begin{proposition} \label{prop:sum-of-two-bands}
	Let $B, C \subseteq X$ be bands.
	\begin{enumerate}[\upshape (a)]
		\item We have $B \cap C = (B^\perp + C^\perp)^\perp$.
		
		\item We have $B+C = (B^\perp \cap C^\perp)^\perp$ if (and only if) $B+C$ is a band.
		
		\item More generally than~\upshape{(b)}, we always have
		\begin{align*}
			B+C \subseteq (B+C)^{\perp \perp} = (B^\perp \cap C^\perp)^\perp.
		\end{align*}
	\end{enumerate}
\end{proposition}
\begin{proof}
	(a) According to Proposition~\ref{prop:intersections-of-arbitrary-many-bands} we have $B \cap C = (B^\perp \cup C^\perp)^\perp$, and the latter set clearly contains $(B^\perp + C^\perp)^\perp$. On the other hand, if $x \in X$ is disjoint to $B^\perp \cup C^\perp$, then it is also disjoint to $B^\perp + C^\perp$ since the disjoint complement of $\{x\}$ is a vector subspace of $X$; this shows that we also have $(B^\perp \cup C^\perp)^\perp \subseteq (B^\perp + C^\perp)^\perp$.
	
	(c) It follows from~(a) that
	\begin{align*}
		B^\perp \cap C^\perp = (B^{\perp \perp} + C^{\perp \perp})^\perp = (B+C)^\perp,
	\end{align*}
	so $(B^\perp \cap C^\perp)^\perp = (B+C)^{\perp\perp}$.
	
	(b) This is an immediate consequence of~(c).
\end{proof}

The main point of the above proposition -- and the reason for the title of this subsection -- is assertion~(b). Anyway, we chose to include assertion~(a) in the same proposition in order to have an immediate comparison between~(a) and~(b).

We point out that the assumption of~(b) that $B+C$ be a band is automatically satisfied of both $B$ and $C$ are projection bands; see Proposition~\ref{prop:sum-of-projection-bands} below. On the other hand, Example~\ref{ex:two-bands-whose-sum-is-not-a-band} shows that there are situations in which $B+C$ is not a band -- and in this case the formula from Proposition~\ref{prop:sum-of-two-bands} necessarily fails.

\section{Band projections} \label{section:band-projections}

Band projections (and, accordingly, projection bands) in pre-Riesz spaces are a main subject of study in \cite{KalauchPreprint4, Glueck2019}. In this section we further develop their theory.

\subsection{Basics} \label{subsection:basics-on-band-projections}

If $B$ is a band in $X$, then it intersects its orthogonal band $B^\perp$ only in $0$. However, the sum of $B$ and $B^\perp$ can be smaller than the entire space $X$, in general; this happens, for instance, in Example~\ref{ex:two-bands-whose-sum-is-not-a-band}, where $C = B^\perp$.

We call a subset $B \subseteq X$ a \emph{projection band} if $B$ is a band and if, in addition, $X = B \oplus B^\perp$. It is not difficult to see that a band $B$ is a projection band if and only if $B^\perp$ is a projection band. Every projection band $B$ has generating cone according to \cite[Proposition~2.5]{Glueck2019}.

The notion of a projection band also gives rise to the following definition: a linear projection $P: X \to X$ is called a \emph{band projection} if there exists a projection band $B \subseteq X$ such that $P$ is the projection onto $B$ along $B^\perp$. In other words, $P$ is a band projection if and only if $PX$ is a projection band and $\ker P$ equals the disjoint complement of $PX$.

The following proposition contains various characterisations of band projections.

\begin{proposition} \label{prop:characterisation-of-band-projections}
	For every linear projection $P: X \to X$ the following assertions are equivalent:
	\begin{enumerate}[\upshape (i)]
		\item $P$ is a band projection.
		\item $\ker P = (PX)^\perp$.
		\item $PX = (\ker P)^\perp$.
		\item $PX \perp \ker P$.
		\item Both projections $P$ and $\id - P$ are positive.
		\item $\id - P$ is a band projection.
	\end{enumerate}
\end{proposition}
\begin{proof}
	``(i) $\Leftrightarrow$ (v)'' This equivalence was proved in \cite[Theorem~3.2]{Glueck2019}. 
	
	``(i) $\Leftrightarrow$ (vi)'' This equivalence follows from the fact that a band $B$ is a projection band if and only if $B^\perp$ is a projection band (alternatively, it follows immediately from the equivalence of (i) and (v)).
	
	``(i) $\Rightarrow$ (ii)'' and ``(i) $\Rightarrow$ (iii)'' These implications follow immediately from the definition of a band projection.
	
	``(ii) $\Rightarrow$ (iv)'' and ``(iii) $\Rightarrow$ (iv)'' These implications are obvious.
	
	``(iv) $\Rightarrow$ (v)'' Let $x \in X_+$. Then $Px$ and $(\id - P)x$ are disjoint and sum up to $x$, so it follows from \cite[Proposition~2.4(a)]{Glueck2019} that $Px$ and $(\id-P)x$ are positive, too. This shows~(v). 
\end{proof}

If $P$ is a band projection in $X$, then both the range and the kernel of $P$ are projection bands. In Corollary~\ref{cor:band-projection-if-range-and-kernel-are-projection-bands} below we will see that the converse implication is also true, which yields another characterisation of band projections.

We conclude this subsection with a few examples.

\begin{examples} \label{ex:examples-for-projection-bands}
	(a) If $X$ is a Dedekind complete vector lattice, then every band in $E$ is a projection band \cite[Theorem~II.2.10]{Schaefer1974}.
	
	(b) Let $(\Omega,\mu)$ be a $\sigma$-finite measure space and let $p \in [1,\infty]$. The bands in $L^p(\Omega,\mu)$ are described in Example~\ref{ex:bands-in-function-spaces}(a). Since $L^p(\Omega,\mu)$ is Dedekind complete, it follows from~(a) that each of these bands is actually a projection band.
	
	(c) The bands $B_a$ in $C([0,1])$ from Example~\ref{ex:bands-in-function-spaces}(b) are not projection bands unless $a = 0$ (see \cite[Example~5 on p.\,63]{Schaefer1974}). More generally, it is not difficult to see that, for a compact Hausdorff space $K$, there are no non-trivial projection bands in $C(K)$ if $K$ is connected.
	
	(d) Let $X$ and $Y$ be two pre-Riesz spaces and endow the product space $Z := X \times Y$ with the product order (i.e.\ $Z_+ = X_+ \times Y_+$). Then $Z$ is a pre-Riesz space, too, and $X$ and $Y$ -- which we identify with the subspaces $X \times \{0\}$ and $\{0\} \times Y$ of $Z$, respectively -- are projection bands in $Z$. Indeed, we have $X^\perp = Y$, and vice versa.
	
	On a related note, we will see in Theorem~\ref{thm:projection-band-structure-of-finite-dimensional-pre-riesz-spaces} below that every finite-dimensional pre-Riesz space can be written as the product of finitely many minimal projection bands.
	
	(e) If $X$ is a Banach lattice and we identify $X$ with a subspace of its bi-dual space $X''$ by means of evaluation, then $X$ is a band in $X''$ if and only if $X$ is a projection band in $X''$ if and only if $X$ is a so-called \emph{KB-space}. This class of space includes all reflexive Banach lattices and all $L^1$-spaces (over arbitrary measure spaces). For details we refer for instance to \cite[Section~2.4]{Meyer-Nieberg1991}.
\end{examples}

Example~\ref{ex:examples-for-projection-bands}(e) can be extended to also include spaces that are not lattice-ordered. An ordered vector space $(Y,Y_+)$ is called an \emph{ordered Banach space} if $Y$ carries a complete norm and $Y_+$ is closed. Note that the order in an ordered Banach space is always Archimedean. Hence, if $Y_+$ is, in addition, generating, then the ordered Banach space $(Y,Y_+)$ is a pre-Riesz space. 

Throughout the rest of the paper, we will tactily use some important concepts from the theory of ordered Banach spaces -- such as \emph{normality} of cones and the fact the the dual of an ordered Banach space with generating cone is again an ordered Banach space. For details about the theory of ordered Banach spaces we refer the reader for instance to the monograph \cite{Aliprantis2007}, in particular to Section~2.5 there.

\begin{example} \label{ex:ordered-banach-space-band-in-bidual}
	Assume that the pre-Riesz space $X$ is an ordered Banach space with normal cone. Then we can consider $X$ as a subspace of the bi-dual space $X''$ by means of evaluation.
	
	There are interesting examples where $X$ is not a Banach lattice and not reflexive, but yet a projection band in $X''$. This is, for instance, the case if $X$ is the pre-dual of a von Neumann algebra; see \cite[Proposition 1.17.7]{Sakai1971} or \cite[pp.\,126--127]{Takesaki1979}.
	
	Ordered Banach space that are projection bands in their bi-dual were employed in \cite{GlueckWolffLB} to study the long-term behaviour of positive operator semigroups.
\end{example}

One can easily find examples where a pre-Riesz space $X$ does not contain any projection bands except for $\{0\}$ and $X$ itself. One situation of this type has already been discussed in Example~\ref{ex:examples-for-projection-bands}(c) above. Here are two more examples.

\begin{examples} \label{ex:non-lattices-without-projection-bands}
	(a) Let us endow $X = \bbR^3$ with the four ray cone $X_+$ from Example~\ref{ex:four-ray-cone-bands}. Then every non-trivial band $B$ in $X_+$ is one-dimensional, so there is no non-trivial projection band in $X$.
	
	(b) Assume that $X$ is a so-called \emph{anti-lattice}, which means that any two vectors $x,y$ in $X$ have a supremum if and only if $x \ge y$ or $x \le y$. Then there are, according to \cite[Theorem~4.1.10(ii)]{Kalauch2019}, no non-trivial disjoint elements in $X_+$. Hence, there are no non-trivial projections bands in $X$.
	
	We note that a classical example of an anti-lattice is the space of all self-adjoint bounded linear operators on a Hilbert space; this result goes back to Kadison \cite[Theorem~6]{Kadison1951}.
\end{examples}

\subsection{The Boolean algebra of band projections} \label{subsection:the-boolean-algebra-of-band-projections}

In this section we study the structure of the collection of all band projections on $X$. As in the vector lattice case, this collection turns out to be a Boolean algebra (Theorem~\ref{thm:band-projections-form-boolean-algebra}).

We begin with the following proposition which shows that a band projection $Q$ dominates a band projection $P$ (in the sense of operators on the ordered vector space $X$) if and only if the range of $Q$ contains the range of $P$:

\begin{proposition} \label{prop:domination-of-band-projections}
	For two band projections $P$ and $Q$ on $X$ the following assertions are equivalent.
	\begin{enumerate}[\upshape (i)]
		\item $PX \subseteq	QX$.
		\item $QP = P$.
		\item $P \le Q$.
	\end{enumerate}
\end{proposition}
\begin{proof}
	``(i) $\Leftrightarrow$ (ii)'' This can immediately be checked to be true for all projections on arbitrary vector spaces.
	
	``(ii) $\Rightarrow$ (iii)'' We have $P = QP \le Q \cdot \id = Q$.
	
	``(iii) $\Rightarrow$ (ii)'' We have $QP \le \id \cdot P = P = P^2 \le QP$, so $QP = P$.
\end{proof}

Next we describe the interaction of two arbitrary band projections on $X$ in a bit more detail; in particular, we prove that any two band projections commute.

\begin{proposition} \label{prop:interaction-of-band-projections}
	For two band projections $P$ and $Q$ on $X$ the following assertions hold:
	\begin{enumerate}[\upshape (a)]
		\item $P$ leaves the range of $Q$ invariant, and vice versa.
		\item $P$ and $Q$ commute.
		\item The mapping $PQ = QP$ is a band projection, too.
		\item We have $PQX = PX \cap QX$.
	\end{enumerate}
\end{proposition}
\begin{proof}
	(a) Let $0 \le x \in QX$. For each $0 \le z \in (QX)^\perp$ it follows from $0 \le Px \le x$ that $Px \perp z$ (see \cite[Proposition~2.2]{Glueck2019}); since the positive cone in the projection band $(QX)^\perp$ is generating in $(QX)^\perp$, we conclude that $Px \perp (QX)^\perp$. 
	
	Now we also use that the positive cone in the projection band $QX$ is generating in $QX$, which implies that $Px \perp (Qx)^\perp$ for each $x \in Qx$. Hence, $Px \in (Qx)^{\perp \perp} = QX$ for each $x \in QX$, which shows that $P$ leaves $QX$ invariant. By interchanging the roles of $P$ and $Q$ we also obtain that $Q$ leaves $PX$ invariant.
	
	(b) It follows from~(a) that $Q$ leaves both $PX$ and $(\id - P)X$ invariant. Thus,
	\begin{align*}
		PQP = QP \qquad \text{and} \qquad PQ(\id - P) = 0.
	\end{align*}
	The second equality is equivalent to $PQP = PQ$ which yields, in conjunction with the first equality, $QP = PQ$.
	
	(c) Clearly, $0 \le PQ \le \id \cdot \id = \id$, so it remains to show that $PQ$ is a projection. Since $P$ leaves $QX$ invariant, we know that $QPQ = PQ$, so $(PQ)^2 = P(QPQ) = P(PQ) = PQ$.
	
	(d) ``$\supseteq$'' For $x \in PX \cap QX$ we have $x = Px = PQx \in PQX$.
	
	``$\subseteq$'' If $x \in PQX$, then $x = PQx \in PX$ and $x = QPx \in QX$.
\end{proof}

We point out that assertion~(d) in the above proposition is in fact true for all commuting projection $P$ and $Q$ on an arbitrary vector space.

As a consequence of the fact that any two band projections commute we obtain the following proposition which shows, in particular, that the sum of two projection bands is a projection band (and a formula for such a sum can thus be found in Proposition~\ref{prop:sum-of-two-bands}(b) above).

\begin{proposition} \label{prop:sum-of-projection-bands}
	Let $P$ and $Q$ be band projections on $X$. Then $P+Q-PQ$ is a band projection, too, and its range coincides with the set $PX + QX$. In particular, the sum of two projection bands is a projection band.
\end{proposition}
\begin{proof}
	Since $P$ and $Q$ commute, a direct computation shows that $P+Q-PQ$ is a projection. Moreover,
	\begin{align*}
		P+Q-PQ = P + Q(\id - P),
	\end{align*}
	and the latter mapping is clearly positive and dominated by $P + \id(\id-P) = \id$. Thus, $P+Q-PQ$ is a band projection.
	
	Obviously, the range of $P+Q-PQ$ is contained in the vector space sum $PX + QX$. The converse inclusion follows from the formula
	\begin{align*}
		Px+Qy = (P+Q-PQ)(Px + Qy)
	\end{align*}
	which holds for all $x,y \in X$.
\end{proof}

Now we can prove that the set of all band projections on $X$ is a Boolean algebra. Recall (for instance from \cite[Definition~II.1.1]{Schaefer1974}) that a \emph{Boolean algebra} is a non-empty partially ordered set $A$ with the following properties:

\begin{enumerate}[(a)]
	\item For all $x,y \in A$ the infimum $x \land y$ and the supremum $x \lor y$ exist (i.e., $A$ is a \emph{lattice}).
	
	\item The lattice operations $\land$ and $\lor$ are \emph{distributive}, i.e., we have
	\begin{align*}
		(x \lor y) \land z = (x \land z) \lor (y \land z)
	\end{align*}
	for all $x,y,z \in A$ (this is equivalent to assuming that $(x \land y) \lor z = (x \lor z) \land (y \lor z)$ for all $x,y,z \in A$, see \cite[Theorem~9 on p.\,11]{Birkhoff1967}).
	
	\item There exists a smallest element $0$ and a largest element $1$ in $A$.
	
	\item $A$ is \emph{complemented}, i.e., for each $x \in A$ there exists a so-called \emph{complement} $x^c \in A$ such that
	\begin{align*}
		x \land x^c = 0 \quad \text{and} \quad x \lor x^c = 1.
	\end{align*}
\end{enumerate}

We note that, in a Boolean algebra $A$, the complement of each element is uniquely determined; this follows from \cite[Theorem~10 on p.\,12]{Birkhoff1967}.

\begin{theorem}	\label{thm:band-projections-form-boolean-algebra}
	Let $\BandProj(X)$ denote the set of all band projections on $X$, ordered by the usual order of positive operators on $X$. Then $\BandProj(X)$ is a Boolean algebra with smallest element $0$ and largest element $\id$. The lattice operations $\land$ and $\lor$ on this Boolean algebra are given by
	\begin{align*}
		P \land Q = PQ \quad \text{and} \quad P \lor Q = P+Q - PQ
	\end{align*}
	for all band projections $P$ and $Q$, and the complement is given by
	\begin{align*}
		P^c = \id - P
	\end{align*}
	for each band projection $P$.
\end{theorem}

In seperable Hilbert spaces ordered by self-dual cones, a related result for the set of all self-adjoint band projections was shown in \cite[Theorem~II.1]{Penney1976} (although the notion \emph{band projection} was not used explicitly there).

In the proof we make use of the facts established in the propositions above; in particular we will frequently -- and often tacitly -- use that $P \le Q$ for two band projections $P$ and $Q$ if and only if $PX \subseteq QX$.

\begin{proof}[Proof of Theorem~\ref{thm:band-projections-form-boolean-algebra}]
	We first show that $\BandProj(X)$ is a lattice with respect to its given order, and that the lattice operators are given by the formulae in the theorem. Let $P,Q \in \BandProj(X)$. 
	
	It follows from Proposition~\ref{prop:interaction-of-band-projections}(d) that $PQ$ is a lower bound of $P$ and $Q$. If $R \in \BandProj(X)$ is another lower bound of $P$ and $Q$, then $RX \subseteq PX$ and $RX \subseteq QX$, so $RX \subseteq PQX$, again by Proposition~\ref{prop:interaction-of-band-projections}(d); hence, $R \le PQ$. This proves that $P$ and $Q$ have infimum $PQ$ in $\BandProj(X)$.
	
	On the other hand, $P+Q-PQ$ is an upper bound of $P$ and $Q$ according to Proposition~\ref{prop:sum-of-projection-bands}.  If $R \in \BandProj(X)$ is another upper bound of $PX$ and $QX$, then $RX \supseteq PX \cup QX$, hence $RX \supseteq PX + QX$ and thus, it follows again from Proposition~\ref{prop:sum-of-projection-bands} that $R \ge P+Q-PQ$. This proves that $P$ and $Q$ have supremum $P+Q-PQ$ in $\BandProj(X)$.
	
	In particular, $\BandProj(X)$ is a lattice. The fact that it is even a distributive lattice, i.e., that the distributive law
	\begin{align*}
		(P \lor Q) \land R = (P \land R) \lor (Q \land R)
	\end{align*}
	is satisfied for all band projections $P,Q,R$, can now be checked by a straightforward computation that uses the formulae for $\land$ and $\lor$ established above.
	
	Clearly, $\BandProj(X)$ has the smallest element $0$ and the largest element $\id$, and for every band projection $P$, the projection $Q := \id - P$ satisfies $P \land Q = PQ = 0$ and $P \lor Q = P + Q - PQ = P+Q = \id$; hence, $I-P$ is the complement of any $P \in \BandProj(X)$ and $\BandProj(X)$ is indeed a Boolean algebra.
\end{proof}

\begin{corollary}
	Let $\ProjBands(X)$ denote the set of all projection bands in $X$, ordered via set inclusion. The mapping
	\begin{align*}
		\varphi: \BandProj(X) & \to \ProjBands(X), \\
		P & \mapsto PX
	\end{align*}
	is an order isomorphism between the partially ordered sets $\BandProj(X)$ and $\ProjBands(X)$. In particular, $\ProjBands(X)$ is a Boolean algebra with infimum and supremum given by
	\begin{align*}
		B \land C = B \cap C \quad \text{and} \quad B \lor C = B+C
	\end{align*}
	for all projections bands $B,C$ in $X$, and with the complement operation given by
	\begin{align*}
		B^c = B^\perp
	\end{align*}
	for each projection band $B$ in $X$.
\end{corollary}
\begin{proof}
	The mapping $\varphi$ is surjective by definition of the notions ``projection band'' and ``band projection'', and it is injective since every band projection $P$ is uniquely determined by its range $PX$. It follows from Proposition~\ref{prop:domination-of-band-projections} that $\varphi$ and its inverse map $\varphi^{-1}$ are monotone. Thus, $\ProjBands(X)$ is indeed a Boolean algebra and $\varphi$ is an isomorphism between the boolean algebras $\BandProj(X)$ and $\ProjBands(X)$.
	
	The formulae for the lattice operations on $\ProjBands(X)$ now follow from Propositions~\ref{prop:interaction-of-band-projections} and~\ref{prop:sum-of-projection-bands}, and the formula for the complement follows from the fact that $(\id - P)X = \ker P = (PX)^\perp$ for each band projection $P$.
\end{proof}

\subsection{The intersection of arbitrarily many projection bands} \label{subsection:the-intersection-of-arbitrarily-many-projection-bands}

According to Proposition~\ref{prop:interaction-of-band-projections}, the intersection of finitely many projection bands is again a projection band. In general, this is no longer true for infinitely many projections bands (not even in the case of Banach lattices) as the following simple example shows:

\begin{example}
	Consider the compact space $K = [-1,0] \cup \{\frac{1}{n}: \; n \in \bbN\}$ and the Banach lattice $C(K)$ of continuous real-valued functions on $K$.
	
	For each $n \in \bbN$ the set $B_n := \{f \in C(K): \; f(x) = 0 \text{ for all } x \ge \frac{1}{n}\}$ is a projection band in $C(K)$. However, the intersection
	\begin{align*}
		\bigcap_{n \in \bbN} B_n 
		& = \{f \in C(K): \; f(x) = 0 \text{ for all } x > 0\} \\
		& = \{f \in C(K): \; f(x) = 0 \text{ for all } x \ge 0\}
	\end{align*}
	is not a projection band in $C(K)$.
\end{example}

However, in a Dedekind complete vector lattice every band is a projection band and hence, the intersection of arbitrarily many projection bands is still a projection band.

Motivated by this we show in this subsection that the intersection of arbitrarily many projection bands in a Dedekind complete pre-Riesz space is again a projection band. Here, we call the pre-Riesz space $X$ \emph{Dedekind complete} if the supremum $\sup A$ exists in $X$ for every non-empty upwards directed set $A \subseteq X$ that is bounded above.

Assume for a moment that $X$ is Dekekind complete. If $(x_j)$ and $(y_j)$ are decreasing nets in $X$ (with the same index set) that are bounded below, then the net $(x_j)$ has an infimum $x$ (we write $x_j \downarrow x$ for this), the net $(y_i)$ has an infimum $y$, and it is not difficult to show that the sum $(x_j + y_j)$ has infimum $x+y$; similarly, for $\lambda \in [0,\infty)$ the net $(\lambda x_j)$ has infimum $\lambda x$.

\begin{theorem} \label{thm:decreasing-net-of-band-projections}
	Assume that $X$ is Dedekind complete and let $(P_j)$ be a net of band projections on $X$ such that $P_j \le P_i$ (equivalently: $P_jX \subseteq P_iX$) whenever $j \ge i$. Then there exists a band projection $P_0$ on $X$ with the following two properties:
	\begin{enumerate}[\upshape (a)]
		\item We have $P_j x \downarrow P_0x$ for each $x \in X_+$.
		\item $P_0X = \bigcap_{j} P_jX$. 
	\end{enumerate}
\end{theorem}
\begin{proof}
	First we define a mapping $P_0: X_+ \to X_+$ by means of $P_0x = \inf_j P_j x$ for each $x \in X_+$. By the remarks we made before the theorem, $P_0$ is linear in the sense that $P_0(\alpha x + \beta y) = \alpha P_0 x + \beta P_0 y$ for all $x,y \in X_+$ and all $\alpha,\beta \in [0,\infty)$. As $X_+$ is generating in $X$, we can extend $P_0$ to a (uniquely determined) linear map -- that we again denote by $P_0$ -- from $X$ to $X$. For each $x \in X_+$ we have $0 \le P_0 x \le x$. 
	
	Let us show next that $P_0$ is a projection; to this end, it suffices to consider $x \in X_+$ and show that $P_0^2x = P_0x$. For each index $j$ we have $0 \le P_0x \le P_j x$, so $P_0x \in P_jX$ and hence, $P_j(P_0x) = P_0x$. This shows that $P_0(P_0x) = P_0x$, so $P_0^2 = P_0$. Consequently, $P_0$ is a band projection that has property~(a). Let us now show~(b).
	
	``$\subseteq$'' Let $x \in P_0X$. Then we can write $x$ as $x = y-z$, where $y$ and $z$ are positive vectors in $P_0X$. For each index $j$, we then have $0 \le y = P_0y \le P_j y$; hence, $y \in P_jX$, and likewise for $z$. Thus, $x = y-z \in P_jX$ for each $j$.
	
	``$\supseteq$'' Let $x \in \bigcap_j P_j X$. We decompose $x$ as $x = y-z$ for $y,z \in X_+$. For each $j$ we have $x = P_j x = P_jy - P_j z$, so 
	\begin{align*}
		x \le P_jy \qquad \text{and} \qquad -x \le P_j z.
	\end{align*}
	Consequently, $x \le P_0y$ and $-x \le P_0 z$, so $x \in [-P_0z, P_0y]$, which proves that $x \in P_0X$.
\end{proof}

In order to derive from Theorem~\ref{thm:decreasing-net-of-band-projections} that the intersection of an arbitrary -- maybe non-directed -- collection of projection bands is still a projection band, we need the following lemma (which is true on every pre-Riesz space, be it Dedekind complete or not).

\begin{lemma} \label{lem:product-of-band-projections-as-infimum}
	Let $P_1,\dots,P_n$ be band projections on $X$ and let $x \in X_+$.
	\begin{enumerate}[\upshape (a)]
		\item If $z \in X$ and $z \le P_1x$, \dots, $z \le P_nx$, then also $z \le P_1 \cdots P_n x$.
		\item We have $P_1\cdots P_n x = \inf\{P_1x,\dots,P_n x\}$.
	\end{enumerate}
\end{lemma}
\begin{proof}
	(a) We first note that, if $z \le Px$ for a band projection $P$, then $(\id - P)z \le 0$. Now we prove the assertion by induction over $n$. For $n=1$ the assertion is obvious, so assume that it has already been proved for some $n \in \bbN$. If $P_{n+1}$ is another band projection such that $z \le P_{n+1}x$, then
	\begin{align*}
		z = (\id - P_{n+1})z + P_{n+1}z \le P_{n+1}z \le P_{n+1}P_1 \cdots P_n x = P_1 \cdots P_{n+1}x.
	\end{align*}
	
	(b) Clearly, $P_1\cdots P_nx$ is a lower bound of $\{P_1x,\dots,P_n x\}$, and according to~(a) it is also the greatest lower bound of this set.
\end{proof}

In the proof of the following corollary we only need assertion~(a) of the lemma. We included assertion~(b) in the lemma anyway since we think it is interesting in its own right.

\begin{corollary} \label{cor:intersection-of-infinitely-many-projection-bands}
	Assume that $X$ is Dedekind complete. Then the intersection of arbitrarily many projection bands in $X$ is again a projection band. More precisely, if $\calP$ is a set of band projections on $X$, then there exists a (unique) band projection $P_0$ on $X$ with range $\bigcap_{P \in \calP} PX$; if $\calP$ is non-empty, then we have
	\begin{align*}
		P_0x = \inf\{Px: \; P \in \calP\}
	\end{align*}
	for each $x \in X_+$
\end{corollary}
\begin{proof}
	We may assume that $\calP$ is non-empty. Let $\hat \calP$ denote the set of all finite products of elements from $\calP$. Then it is easy to see that $\bigcap_{P \in \calP}PX = \bigcap_{P \in \hat \calP} PX$. Moreover, $\hat \calP$ is directed by the converse of the usual order $\le$ on linear operators (since $\hat \calP$ is closed with respect to taking finite products). Thus, $(P)_{P \in \hat \calP}$ is a decreasing net of band projections, so Theorem~\ref{thm:decreasing-net-of-band-projections} shows the existence of a band projection $P_0$ on $X$ such that $P_0X = \bigcap_{P \in \hat\calP} PX$.
	
	It remains to prove the formula for $P_0x$, so let $x \in X_+$. By Theorem~\ref{thm:decreasing-net-of-band-projections}(a) we have
	\begin{align*}
		P_0x = \inf\{Px: \; P \in \hat \calP\}.
	\end{align*}
	Clearly, $P_0x$ is a lower bound of $\{Px: \; P \in \calP\}$, so let $z \in X$ be another lower bound of this set. Then $z$ is, according to Lemma~\ref{lem:product-of-band-projections-as-infimum}(a), also a lower bound of $\{Px: \; P \in \hat \calP\}$, and hence $z \le P_0x$.
\end{proof}

If $X$ is Dedekind complete, then it follows from Corollary~\ref{cor:intersection-of-infinitely-many-projection-bands} that, for every set $S \subseteq X$, there exists a smallest projection band that contains $S$. This projection band can, however, be much larger than the band generated by $S$, as the following example shows:

\begin{example}
	Let $X = \mathbb{R}^3$, let $X_+$ be the four ray cone from Example~\ref{ex:four-ray-cone-bands} and let $v_1$ by the vector introduced in that example. The span of $\{v_1\}$ is a band, but according to Example~\ref{ex:non-lattices-without-projection-bands}(a) there are no non-trival projection bands in $X$. 
	
	Hence, the band generated by $\{v_1\}$ equals $\operatorname{span}\{v_1\}$, while the projection band generated by $\{v_1\}$ equals $X$.
\end{example}

\subsection{Another characterisation of band projections} \label{subsection:another-characterisation-of-band-projections}

The following propositions shows that if two projections bands $B$ and $C$ have trivial intersection, than we automatically have $B \perp C$.

\begin{proposition} \label{prop:projection-band-with-trivial-intersection}
	For two band projections $P$ and $Q$ on $X$ the following assertions are equivalent.
	\begin{enumerate}[\upshape (i)]
		\item $PQ = 0$
		\item $PX \cap QX = \{0\}$.
		\item $PX \perp QX$.
	\end{enumerate}
\end{proposition}
\begin{proof}
	``(i) $\Leftrightarrow$ (ii)'' This equivalence follows from Proposition~\ref{prop:interaction-of-band-projections}(d) (and is thus true for arbitrary commuting projections on every vector space).
	
	``(i) $\Rightarrow$ (iii)'' According to Proposition~\ref{prop:orthogonal-subspaces-by-positive-parts} it suffices to show that $PX_+ \perp QX_+$, so let $x \in PX_+$ and $y \in QX_+$. In order to show that $x \perp y$ it is necessary and sufficient to prove that $x$ and $y$ have infimum $0$ in $X$. Obviously, $0$ is a lower bound of $x$ and $y$, so let $b$ be another lower bound of those vectors. We then have 
	\begin{align*}
		Pb \le Py = PQy = 0 \quad \text{and} \quad Qb \le Qx = QPx = 0,
	\end{align*}
	so $(P+Q)b \le 0$. On the other hand, we know from Proposition~\ref{prop:sum-of-projection-bands} that $P+Q$ is a band projection (since $PQ = 0$), so $\id - (P+Q)$ is positive. Hence,
	\begin{align*}
		(\id - (P+Q))b \le (1 - (P+Q))x = x - Px - Qx = -Qx = 0.
	\end{align*}
	Consequently, $b = (\id - (P+Q))b + (P+Q)b \le 0$. This proves that $x$ and $y$ indeed have infimum $0$.
	
	``(iii) $\Rightarrow$ (ii)'' For each $x \in PX \cap QX$ we have $x \perp x$, so $x = 0$.
\end{proof}

We remark that the implication ``(iii) $\Rightarrow$ (ii)'' in Proposition~\ref{prop:projection-band-with-trivial-intersection} remains true if $PX$ and $QX$ are replaced with arbitrary bands (over even arbitrary vector subspaces) in $X$. However, the converse implication fails for general bands, even if they are assumed to be directed. We illustrated this, again, in the space $\mathbb{R}^3$ endowed with the four ray cone.

\begin{example} \label{ex:non-disjoint-bands-with-trivial-intersection}
	Let $X = \mathbb{R}^3$, let $X_+$ denote the four ray cone from Example~\ref{ex:four-ray-cone-bands}, and let $v_1,\dots, v_4$ denote the four vectors defined in the same example.
	
	Then $B_1 := \operatorname{span}\{v_1\}$ and $B_2 := \operatorname{span} \{v_2\}$ are bands in $X$ that intersect only in $0$. However, we do not have $B_1 \perp B_2$ since $v_1$ is not disjoint to $v_2$. To see this, consider the vector
	\begin{align*}
		w =
		\begin{pmatrix}
			1 \\ 1 \\ 0
		\end{pmatrix}.
	\end{align*}
	Then $v_1-w = v_4 \in X_+$ and $v_2-w = v_3 \in X_+$. Hence, $w$ is a lower bound of both $v_1$ and $v_2$. On the other hand, $w$ is not an element of the negative cone $-X_+$. Thus, $0$ is not the greatest lower bound of $v_1$ and $v_2$, so $v_1 \not\perp v_2$.
\end{example}

As a consequence of Proposition~\ref{prop:projection-band-with-trivial-intersection} we obtain another characterisation of band projections.

\begin{corollary} \label{cor:band-projection-if-range-and-kernel-are-projection-bands}
	For every linear projection $P: X \to X$ the following assertions are equivalent:
	\begin{enumerate}[\upshape (i)]
		\item $P$ is a band projection.
		\item Both $PX$ and $\ker P$ are projection bands.
	\end{enumerate}
\end{corollary}
\begin{proof}
	``(i) $\Rightarrow$ (ii)'' If $P$ is a band projection, then $PX$ is a projection band by definition, and hence $\ker P = (PX)^\perp$ is also a projection band.
	
	``(ii) $\Rightarrow$ (i)'' If $PX$ and $\ker P$ are projection bands, than there exist band projections $Q_1,Q_2: X \to X$ such that $Q_1X = PX$ and $Q_2X = \ker P$. Since $Q_1X \cap Q_2X = \{0\}$, it follows from Proposition~\ref{prop:projection-band-with-trivial-intersection} that $Q_1X \perp Q_2X$, i.e., $PX \perp \ker P$. According to Proposition~\ref{prop:characterisation-of-band-projections} this implies that $P$ is a band projection.
\end{proof}

We note that the implication ``(ii) $\Rightarrow$ (i)'' in Corollary~\ref{cor:band-projection-if-range-and-kernel-are-projection-bands} does not remain true, in general, if we replace ``projection bands'' in (ii) with ``bands''. More precisely, we have the following situation:

\begin{remarks}
	\begin{enumerate}[(a)]
		\item There exists a (\emph{weakly pervasive}, see Definition~\ref{def:pervasive-and-weakly-pervasive}) pre-Riesz space $X$ and two bands $B$ and $C$ in $X$ such that $X = B \oplus C$, but $C \not= B^\perp$. A concrete example of this situation can be found in \cite[Example~19]{KalauchPreprint4}; it is, however, important to observe that one of the bands is not directed in this example.
		
		\item If $X$ is weakly pervasive and $X = B \oplus C$ for two directed bands -- or, more generally, two directed ideals -- $B$ and $C$, then it is shown in \cite[Theorem~18]{KalauchPreprint4} that $B$ and $C$ are projections bands and $B = C^\perp$.
		
		\item If $X$ is even \emph{pervasive} (see Definition~\ref{def:pervasive-and-weakly-pervasive}), then the implication mentioned in~(b) remains true even if $B$ and $C$ are only ideals in $X$ (which are not a priori assumed to be directed); this is shown in \cite[Theorem~17]{KalauchPreprint4}.
		
		\item Now, let $X$ be a general pre-Riesz space and let $X = B\oplus C$ for two directed bands -- or, more generally, directed ideals -- $B$ and $C$. It seems to be open whether this implies $C = B^\perp$.
	\end{enumerate}
\end{remarks}

\section{Characterisations of vector lattices} \label{section:characterisations-of-vector-lattices}

In this section we give various criteria for a pre-Riesz space to actually be a vector lattice. All these criteria are in some way related to disjointness. We note that, in the important special case where $X$ is finite dimensional and Archimedean, several sufficient criteria for $X$ to be a vector lattice are known. It suffices, for instance, if $X$ has the Riesz decomposition property (see for instance \cite[Corollary~2.48]{Aliprantis2007}) or if $X$ is pervasive \cite[Theorem~39]{KalauchPreprint4}. In Corollary~\ref{cor:finite-dimensional-weakly-pervasive-spaces} below we give a simultaneous generalisation of those two results.

\subsection{Criteria in terms of one-dimensional projection bands} \label{subsection:criteria-in-terms-of-one-dimensional-projection-bands}

In this subsection we prove that a finite-dimensional pre-Riesz space is automatically a vector lattice if there exist sufficiently many projection bands in it. We begin with the following proposition about linear independence.

\begin{proposition} \label{prop:disjoint-vectors-are-linearly-independent}
	Let $m \in \bbN$ and let $x_1,\dots,x_m \in X \setminus \{0\}$ be pairwise disjoint. Then the tuple $(x_1,\dots,x_m)$ is linearly independent.
\end{proposition}
\begin{proof}
	For $m = 1$ the assertions is obvious, and we next show it for $m = 2$. So let $\alpha_1 x_1 + \alpha_2 x_2 = 0$ for real numbers $\alpha_1,\alpha_2$. Since the sum of the disjoint vectors $\alpha_1 x_1$ and $\alpha_2 x_2$ is both positive and negative, it follows from \cite[Proposition~2.4(a)]{Glueck2019} that both vectors $\alpha_1 x_1$ and $\alpha_2 x_2$ are both positive and negative, and thus $0$. This implies that $\alpha_1 = \alpha_2 = 0$ since $x_1,x_2 \not= 0$ by assumption.
	
	Now assume that the assertion has been proved for a fixed integer $m \ge 2$ and let $\alpha_1, \dots, \alpha_{m+1} \in \bbR$ such that $\sum_{k=1}^{m+1} \alpha_k x_k = 0$. Since the vectors $\sum_{k=1}^m \alpha_k x_k$ and $x_{m+1}$ are disjoint and linearly dependent, it follows from the case $m=2$ considered above that $\sum_{k=1}^m \alpha_k x_k = 0$. Using that the assertion has already been proved for the number $m$, we conclude that $\alpha_1 = \dots = \alpha_m = 0$. Finally, we observe that $\alpha_{m+1}x_{m+1} = 0$, so $\alpha_{m+1} = 0$ since $x_{m+1} \not= 0$.
\end{proof}

\begin{theorem} \label{thm:n-distinct-band-projections-with-rank-1}
	Assume that $n = \dim X < \infty$. If there exist (at least) $n$ distinct band projections of rank $1$ on $X$, then $X$ is an Archimedean vector lattice, i.e., $X$ is linearly order isomorphic to $\bbR^n$ with the standard cone.
\end{theorem}
\begin{proof}
	We may assume that $n \not= 0$. Let $P_1,\dots,P_n$ denote $n$ distinct band projections of rank $1$. Then we have $P_k P_jX = \{0\}$ for $j \not= k$. Indeed, if we assumed $\dim (P_k P_jX) = 1$, then $P_kP_j X = P_kX = P_jX$ -- which would imply $P_k = P_j$ since band projections are uniquely determined by their range.
	
	Since each projection band is spanned by its positive elements, each space $P_kX$ is spanned by a vector $x_k > 0$. According to Proposition~\ref{prop:projection-band-with-trivial-intersection} the vectors $x_1,\dots,x_n$ are pairwise disjoint. Hence, they are linearly independent by Proposition~\ref{prop:disjoint-vectors-are-linearly-independent}. Since $\dim X = n$, this implies that the vectors $x_1,\dots,x_n$ span $X$.
	
	It follows from Proposition~\ref{prop:projection-band-with-trivial-intersection} that $P_k P_j = 0$ whenever $j \not= k$, so
	\begin{align*}
		(P_1+\dots+P_n)x_k = x_k \qquad \text{for each } k \in \{1,\dots,n\};
	\end{align*}
	hence, $P_1 + \dots + P_n = \id$, and we conclude that the linear mapping
	\begin{align*}
		J: \; Y:= P_1X \times \dots \times P_nX \ni (z_1,\dots,z_n) \to z_1+\dots+z_n \in X
	\end{align*}
	is a bijection. Each $P_kX$ is an ordered space with respect to the order inherited from $X$, and as such it is isomorphic to $\bbR$ with the cone $[0,\infty)$. If we endow $Y$ with the product order, then $Y$ is isomorphic to $\bbR^n$ with the standard cone, and the mapping $J$ is an order isomorphism between $Y$ and $X$, which proves the assertion.
\end{proof}

There is a certain conceptual similarity between the above proof and the approach taken in \cite[Section~6]{KalauchPreprint4} to prove \cite[Theorem~39]{KalauchPreprint4}: the authors of \cite{KalauchPreprint4} prove that every finite dimensional Archimedean pervasive pre-Riesz space is actually a vector lattice by considering atoms in such a space and by showing that if atoms $a_1,\dots,a_m$ in a pervasive pre-Riesz space are pairwise linearly independent, then the entire system $(a_1,\dots,a_m)$ is linearly independent. Our usage of Proposition~\ref{prop:disjoint-vectors-are-linearly-independent} and of rank-$1$ band projections in the above proof is somewhat reminiscent of this approach (as it is easy to see that the range of a rank-$1$ band projection is always spanned by an atom).

As a simple consequence of Theorem~\ref{thm:n-distinct-band-projections-with-rank-1} we obtain the following numerical bound on the number of rank-$1$ band projections in $X$:

\begin{corollary} \label{cor:n-distinct-one-dimensional-projection-bands}
	Assume that $n = \dim X < \infty$. Then there exists at most $n$ distinct band projections of rank $1$ on $X$.
\end{corollary}
\begin{proof}
	If there were strictly more than $n$ distinct band projections of rank $1$, then Theorem~\ref{thm:n-distinct-band-projections-with-rank-1} would imply that $X$ is isomorphic to $\bbR^n$ with the standard cone -- but on this space there exist precisely $n$ distinct rank-$1$ band projections, so we arrive at a contradiction.
\end{proof}

\subsection{Criteria in terms of the number of projection bands} \label{subsection:criteria-in-terms-of-the-number-of-projection-bands}

If the space $X$ is finite-dimensional and has closed cone, then it follows from~\cite[Theorem~4.4.26]{Kalauch2019} that there exist only finitely many bands in $X$. In particular, the Boolean algebra $\ProjBands(X)$ is finite, so we conclude that the number of projection bands in $X$ is a power $2^m$ of $2$. In the following we are going to prove a bit more: we will not assume $X_+$ to be closed a priori, we will show that we always have $m \le \dim X$, and that equality holds if and only if $X$ is an Archimedean vector lattice.

Let us start with the following slightly more sophisticated version of Proposition~\ref{prop:disjoint-vectors-are-linearly-independent}.

\begin{proposition} \label{prop:disjoint-vectors-are-linearly-independent-more-sophisticated}
	Let $A_1, \dots, A_m \subseteq X$ be subsets of $X$ such that $A_i \perp A_j$ whenever $i\not= j$. For each $j \in \{1,\dots,m\}$, let $(x_{j,1},\dots,x_{j,n_j})$ be a linearly independent system of vectors in $A_j$. Then the entire system
	\begin{align*}
		(x_{1,1},\dots,x_{1,n_1}, \dots, x_{m,1}, \dots, x_{m,n_m})
	\end{align*}
	is linearly independent.
\end{proposition}
\begin{proof}
	First we note that, for any two distinct indices $i,j \in \{1,\dots,m\}$, we have $\linSpan(A_i) \perp A_j$ and hence $\linSpan(A_i) \perp \linSpan(A_j)$ since $X$ is a pre-Riesz space. Now assume that
	\begin{align*}
		\sum_{j=1}^m \sum_{k=1}^{n_j} \lambda_{j,k} x_{j,k} = 0
	\end{align*}
	for scalars $\lambda_{j,k} \in \bbR$. We define vectors $y_j = \sum_{k=1}^{n_j} \lambda_{j,k} x_{j,k} \in \linSpan(A_j)$ for $j \in \{1,\dots,m\}$. Thus, the vectors $y_1,\dots,y_m$ are pairwise disjoint. Since
	\begin{align*}
		y_1 + \dots + y_m = 0,
	\end{align*}
	it follows from Proposition~\ref{prop:disjoint-vectors-are-linearly-independent} that one of these vectors is $0$, and inductively we then derive that actually all vectors $y_1,\dots,y_m$ are $0$.
	
	Now, fix $j \in \{1,\dots,m\}$. Since $0 = y_j = \sum_{k=1}^{n_j} \lambda_{j,k} x_{j,k}$, we conclude from the linear independence of the system $(x_{j,1},\dots,x_{j,n_j)}$ that $\lambda_{j,1} = \dots = \lambda_{j,n_j} = 0$. This proves the assertion.
\end{proof}

A second ingredient that we need is the following simple observation about band projections.

\begin{lemma} \label{lem:sum-of-band-projections}
	Let $P_1,\dots,P_m: X \to X$ be band projections and assume that $P_i P_j = 0$ whenever $i\not= j$. Then
	\begin{align*}
		P_1 + \dots + P_m
	\end{align*}
	is also a band projection.
\end{lemma}
\begin{proof}
	The assumptions clearly imply that $P_1 + \dots + P_m$ is a positive projection. Next, we show by induction over $m$ that
	\begin{align*}
		\id - (P_1 + \dots + P_m) = (\id - P_1) (\id - P_2) \cdots (\id - P_m).
	\end{align*}
	For $m = 1$ this is obvious, so assume that it has been proved for some fixed $m \in \bbN$ and consider now one more band projection $P_{m+1}$ such that $P_jP_{m+1} = 0$ for all $j \in \{1,\dots,m\}$. Then 
	\begin{align*}
		(\id - P_1) (\id - P_2) \cdots (\id - P_m)(\id - P_{m+1}) & = \big(\id - (P_1 + \dots + P_m)\big)(\id - P_{m+1}) \\
		& = \id - (P_1 + \dots + P_{m+1}),
	\end{align*}
	as claimed. We thus conclude that $\id - (P_1 + \dots + P_m)$ is positive, too.
\end{proof}

Now we can prove the first main result of this subsection.

\begin{theorem} \label{thm:only-finitely-many-band-projections}
	If $n = \dim X < \infty$, the following assertions hold:
	\begin{enumerate}[\upshape (a)]
		\item The number of band projections on $X$ is equal to $2^m$ for some $m \in \{0,\dots,n\}$.
		
		\item We have $m = n$ if and only if $X$ is an Archimedean vector lattice.
	\end{enumerate}
\end{theorem}

It is hardly surprising that the proof of Theorem~\ref{thm:only-finitely-many-band-projections} below is strongly related to the Boolean algebra structure of the set of all projection bands. However, we cannot rely on this Boolean structure alone since we want to relate the number $m$ to the dimension of $X$ -- i.e., we need to take the linear structure of the underlying space into account.

\begin{proof}[Proof of Theorem~\ref{thm:only-finitely-many-band-projections}]
	We may assume throughout the proof that $n \not= 0$.
	
	(a) \emph{Step 1:} Within this proof, let us call a projection band $B$ minimal if it is non-zero and if it does not contain any non-zero projection band except itself. Since $X$ is finite dimensional, every non-zero projection band contains a minimal projection band. Let $\calM$ denote the set of all minimal projection bands in $X$. If $B,C \in \calM$ are two distinct projection bands, then $B \cap C = \{0\}$; indeed, $B \cap C$ is a projection band that is contained in both $B$ and $C$. Hence, if it were non-zero, we would have $B \cap C = B$ and $B \cap C = C$, so $B = C$.
	
	Consequently, $B \perp C$ for any two distinct $B,C \in \calM$ by Proposition~\ref{prop:projection-band-with-trivial-intersection}. It thus follows from Proposition~\ref{prop:disjoint-vectors-are-linearly-independent} that there exist at most $n$ distinct minimal projection bands in $X$; we enumerate them as $B_1,\dots,B_m$ (where $1 \le m \le n$), and we denote the corresponding band projections by $P_1,\dots,P_m$.
	
	Since $P_iP_j = 0$ whenever $i \not= j$, it follows from Lemma~\ref{lem:sum-of-band-projections} that $P_1 + \dots + P_m$ is a band projection. Actually, this band projection coincides with $\id$, since otherwise the range of the complementary band projection $\id - (P_1 + \dots + P_m)$ would contain one of the minimal projection bands $B_1,\dots,B_m$, which is a contradiction. Hence,
	\begin{align*}
		P_1 + \dots + P_m = \id.
	\end{align*}
	Consequently, $B_1 + \dots + B_m = X$.
	
	\emph{Step 2:} Next we note that, for each projection band $C$ in $X$ and each $k \in \{1,\dots,m\}$ we have either $B_j \subseteq C$ or $B_j \cap C = \{0\}$; this is a consequence of the minimality of $B_j$. Hence, for every band projection $P$ on $X$ and every $j \in \{1,\dots,m\}$ we have either $PP_j = P_j$ or $P P_j = 0$.
	
	Thus, for every band projection $P$ on $X$ we have
	\begin{align*}
		P = \sum_{j\in I_P} P_j,
	\end{align*}
	where $I_P := \big\{j \in \{1,\dots,m\}: \, P_jP \not= 0\big\}$. Conversely, we note that the sum $P_I := \sum_{j\in I} P_j$ is, for any $I \subseteq \{1,\dots,m\}$, a band projection (according to Lemma~\ref{lem:sum-of-band-projections}), and the set $I$ is uniquely determined by this sum (since it is the set of all $k$ such that $P_kP_I \not= 0$). This proves that there exist exactly $2^m$ band projections on $X$, and we have already observed above that $m \le n$. We have thus proved~(a)
	
	(b) Assume now that $m = n$.
	
	For every $j \in \{1,\dots,m\}$ we now choose a basis $(x_{j,1},\dots, x_{j,n_j})$ of the space $B_j$. It follows from Proposition~\ref{prop:disjoint-vectors-are-linearly-independent-more-sophisticated} that the system
	\begin{align*}
		(x_{1,1},\dots,x_{1,n_1}, \dots, x_{m,1}, \dots, x_{m,n_m})
	\end{align*}
	is linearly independent. Hence, $n_1 + \dots + n_m \le n$. As $m = n$, it follows that none of the numbers $n_j$ can be larger than $1$, so each of the $n$ projection bands $B_1,\dots, B_m = B_n$ is one-dimensional. Theorem~\ref{thm:n-distinct-band-projections-with-rank-1} thus shows that $X$ is an Archimedean vector lattice.
	
	Conversely, if $X$ is an Archimedean vector lattice, then it is isomorphic to $\bbR^n$ with the standard cone, so there exist indeed $2^n$ band projections on $X$, so $m=n$.
\end{proof}

Step~1 in the proof of Theorem~\ref{thm:only-finitely-many-band-projections}(a) also provides us with another interesting insight into the structure of finite-dimensional pre-Riesz spaces. The facts that $P_iP_j = 0$ for any two distinct $i,j \in \{1,\dots,m\}$ and that $P_1 + \dots + P_m = \id$ imply that the mapping
\begin{align*}
	X & \to B_1 \times \dots \times B_m, \\
	x & \mapsto (P_1x,\dots,P_mx)
\end{align*}
is an isomorphism of ordered vector spaces, where $B_1 \times \cdots \times B_m$ is endowed with the product order. 
Moreover, it is not difficult to see that every projection band in a pre-Riesz space is itself a pre-Riesz space; hence, each $B_j$ is a pre-Riesz space.

We also observe that none of the pre-Riesz spaces $B_j$ contains a non-trivial projection band. Indeed, if $Q: B_j \to B_j$ is a band projection, then $QP_j: X \to X$ is a band projection with the same range as $Q$; by the minimality if $B_j$ this implies that this range is either $\{0\}$ or $B_j$. We thus have the following structure result, which is the second main result of this subsection.

\begin{theorem} \label{thm:projection-band-structure-of-finite-dimensional-pre-riesz-spaces}
	Assume that $1 \le \dim X < \infty$. Then there exists a number $m \in \{1,\dots,\dim X\}$ such that $X$ is isomorpic (as an ordered vector space) to the product of $m$ non-zero pre-Riesz spaces none of which contains a non-trivial projection band.
\end{theorem}

For Hilbert spaces ordered by self-dual cones a related structure result (even in infinite dimensions) can be found in \cite[Corollary~II.12]{Penney1976}.

\subsection{Criteria in the class of weakly pervasive spaces} \label{subsection:criteria-in-the-class-of-weakly-pervasive-spaces}

Assume for a moment that $X$ is finite dimensional with closed cone. In \cite[Theorem~39]{KalauchPreprint4} it was shown that if $X$ is pervasive (see Definition~\ref{def:pervasive-and-weakly-pervasive} below), then $X$ is in fact a vector lattice. The same is true if $X$ is assumed to have the Riesz decomposition property instead if being pervasive (see for instance \cite[Corollary~2.49]{Aliprantis2007}). These observations suggest to study the following two questions:

\begin{enumerate}[(a)]
	\item Since the Riesz decomposition property and the property of beging pervasive are logically independent for general pre-Riesz spaces (see \cite[Example~13]{Kalauch2019a} and \cite[Example~23]{Malinowski2018}), it is natural to seek for a simultaneous generalisations of the two above mentioned results. 
	
	\item The fact that the Riesz decomposition property implies that $X$ is a vector lattice is actually not only true in finite-dimensional spaces (with closed cone), but for instance also for the more general case of reflexive ordered Banach spaces with generating cone \cite[Corollary~2.48]{Aliprantis2007}. This suggests that searching for sufficient criteria for infinite dimensional spaces to be a vector lattice is a worthwhile endeavour. 
\end{enumerate}

In this subsection we pursue both goals outlined above. As before, we assume that $X$ is a general pre-Riesz space.

Two vectors $x,y \in X_+$ are called \emph{D-disjoint} if $[0,x] \cap [0,y] = \{0\}$. For a more detailed discussion of this notion and of its origin, we refer to \cite[Section~4.1.3]{Kalauch2019} and \cite[Defintion~8 and Proposition~9]{Katsikis2006}.

Every two disjoint elements in $X$ are clearly D-disjoint, but the converse implication is not true, in general; this can, for instance, again be seen by considering the four ray cone in $\bbR^3$:

\begin{example} \label{ex:D-disjoint-does-not-imply-disjoint}
	Let $X = \mathbb{R}^3$ and let $X_+$ denote the four ray cone from Example~\ref{ex:four-ray-cone-bands}; let $v_1$ and $v_2$ denote the vectors given in the same example. According to Example~\ref{ex:non-disjoint-bands-with-trivial-intersection} the vectors $v_1$ and $v_2$ are not disjoint. However, both elements $v_1$ and $v_2$ are so-called \emph{atoms} in $X$ (see \cite[Definition~27 and Proposition~28]{KalauchPreprint4} or Subsection~\ref{subsection:criteria-in-terms-of-other-concepts-of-disjointness} below), so it follows that $v_1$ and $v_2$ are $D$-disjoint.
\end{example}

Hence, disjointness of two vectors $x,y \in X_+$ is, in general, a much stronger property than D-disjointness. There are, however, spaces in which both notions coincide; this gives rise to part (a) of the following definition.

\begin{definition} \label{def:pervasive-and-weakly-pervasive}
	\begin{enumerate}[(a)]
		\item The pre-Riesz space $X$ is called \emph{weakly pervasive} if any two D-disjoint vectors in $X_+$ are disjoint.
		
		\item The pre-Riesz space $X$ is called \emph{pervasive} if for every $b \in X$ such that $b \not\le 0$ there exists $x \in X_+ \setminus \{0\}$ such that every positive upper bound of $b$ is also an upper bound of $x$.
	\end{enumerate}
\end{definition}

The concept of a weakly pervasive pre-Riesz space was coined in \cite[Definition~8 and Lemma~9]{Kalauch2019a}. The usual definition of a pervasive pre-Riesz space in the literature is somewhat different and employs the Riesz completion of $X$ (see \cite[Definition~2.8.1]{Kalauch2019}). However, this definition is equivalent to the one given above according to \cite[Theorem~7]{Kalauch2019a}.

If one uses that two vectors $x,y \in X_+$ are disjoint if and only if they have infimum $0$, it is easy to show that every pervasive pre-Riesz space is also weakly pervasive. Moreover, every vector lattice is pervasive and hence weakly pervasive.

We also note that every pre-Riesz space with the Riesz decomposition property is weakly pervasive \cite[Proposition~11]{Kalauch2019a}; hence, weakly pervasive spaces are a simultaneous generalisation of pervasive pre-Riesz spaces and pre-Riesz spaces with the Riesz decomposition property.

Let us give a simple criterion in order to check that several function spaces are pervasive.

\begin{proposition} \label{prop:subspaces-of-continuous-functions-which-are-weakly-pervasive}
	Let $\Omega \subseteq \bbR^d$ be open, let $\Omega \subseteq L \subseteq \overline{\Omega}$ and let $X$ be a directed vector subspace of $C(L)$ (where $C(L)$ denotes the space of all real-valued continuous functions on $L$). Then $X$ is a pre-Riesz space; if, in addition, $X$ contains all test functions on $\Omega$, then $X$ is pervasive.
\end{proposition}
\begin{proof}
	As $C(L)$ is Archimedean, so is $X$, and since $X_+$ is generating in $X$ by assumption, it follows that $X$ is a pre-Riesz space. 
	
	Now assume that $X$ contains all test functions on $\Omega$. Let $b \in X$ and $b \not\le 0$. Since $C(L)$ is a vector lattice, we can take the positive part $b^+$ in $C(L)$. This is a non-zero positive continuous function on $L$, so there exists a positive non-zero test function $x$ on $\Omega$ such that $x \le b^+$. We note that $x \in X$ by assumption. Now, if $u \in X_+$ is an upper bound of $b$ in $X$, then it is also an upper bound of $b^+$ in $C(L)$. Hence, $u \ge x$.
\end{proof}

As a consequence of the above proposition we obtain, for instance, the following examples of pervasive spaces.

\begin{examples} \label{ex:list-of-weakly-pervasive-spaces}
	(a) Let $\emptyset \not= \Omega \subseteq \bbR^d$ be open and bounded and let $k \in \mathbb{N}_0$. Then the space $C_b^k(\overline{\Omega})$ of functions that are $k$-times continuously differentiable on $\Omega$ and whose partial derivatives up to order $k$ all have a continuous extension to $\overline{\Omega}$ is pervasive.
	
	(b) Let $\emptyset \not= \Omega \subseteq \bbR^d$ be open and bounded with Lipschitz boundary, let $p \in [1,\infty]$ and $k \in \bbN$ such that $kp > d$. Then the positive cone in the Sobolev space $W^{k,p}(\Omega)$ is closed (with respect to the usual Sobolev norm), and it is also generating (see \cite[Examples~2.3(c) and~(d)]{Arendt2009}); hence, $W^{k,p}(\Omega)$ is a pre-Riesz space. Moreover, $W^{k,p}(\Omega)$ is also pervasive since it embeds into $C(\Omega)$ and since it contains all test functions on $\Omega$.
\end{examples}

In \cite[Example~13]{Kalauch2019a} one can find an example of a pre-Riesz space that is not pervasive, but has the Riesz decomposition property and is thus weakly pervasive.

We now prove the main result of this subsection; it gives a sufficient criterion for a weakly pervasive pre-Riesz space to already be a vector lattice.

\begin{theorem}	\label{thm:weakly-pervasive-monotonically-complete-spaces-are-lattices}
	Assume that every non-empty totally ordered subset of $X$ that is bounded from above has a supremum. If $X$ is weakly pervasive, then $X$ is a lattice.
\end{theorem}
\begin{proof}
	If suffices to show that any two positive elements in $X$ have an infimum, so let $x,y \in X_+$. It follows from Zorn's lemma and from the assumption on $X$ that the set $[0,x] \cap [0,y]$ has a maximal element $a$. Let us show that $[0,x-a] \cap [0,y-a] = \{0\}$: if $z$ is an element of this set, then $0 \le z \le x-a$ and $0 \le z \le y-a$, so $0 \le z+a \le x$ and $0 \le z+a \le y$. Hence, $z+a$ is an element of $[0,x] \cap [0,y]$ that dominates $a$; it thus follows from the maximality of $a$ that $z = 0$.
	
	As $X$ is weakly pervasive, this implies that the positive vectors $x-a$ and $y-a$ are disjoint, i.e.\ they have infimum $0$. Consequently, $x$ and $y$ also have an infimum (namely $a$).
\end{proof}

In the context of ordered Banach spaces, the following proposition gives sufficient criteria for the first assumption of Theorem~\ref{thm:weakly-pervasive-monotonically-complete-spaces-are-lattices} to be satisfied.

\begin{proposition} \label{prop:regularity-properties-of-ordered-banach-spaces}
	Assume that the pre-Riesz space $X$ is an ordered Banach space. Consider the following assertions:
	\begin{enumerate}[\upshape (i)]
		\item The cone $X_+$ is normal and the space $X$ is reflexive.
		
		\item The cone $X_+$ is normal and $X$ is a projection band in its bi-dual (compare Example~\ref{ex:ordered-banach-space-band-in-bidual}).
		
		\item Every order interval in $X$ is weakly compact.
		
		\item $X$ is the dual space of an ordered Banach space $Y$ such that $Y$ has generating cone.
		
		\item The norm is additive on $X_+$ (i.e., $\lVert x+y\rVert = \lVert x \rVert + \lVert y \rVert$ for all $x,y \in X_+$).
		
		\item Every increasing norm bounded net in $X_+$ is norm convergent.
		
		\item The cone $X_+$ is normal and every increasing net in $X_+$ that is bounded from above is norm convergent.
		
		\item Every non-empty upwards directed set in $X$ that is bounded above has a supremum (i.e.,\ in the terminology of Subsection~\ref{subsection:the-intersection-of-arbitrarily-many-projection-bands}, $X$ is Dedekind complete).
		
		\item Every non-empty totally ordered set in $X$ that is bounded above has a supremum.
	\end{enumerate}
	Then the following implications hold: \\
	
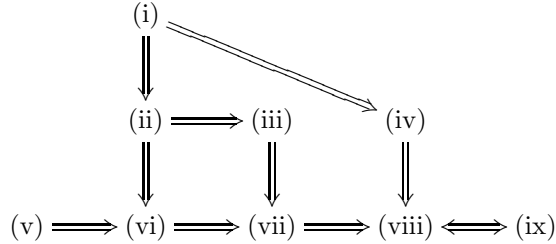
\centerline{
		\xymatrix{
			&
			{\text{\upshape(i)}}
			\ar@{=>}[d]
			\ar@{=>}[drr]
			&
			&
			&
			\\
			&
			{\text{\upshape(ii)}}
			\ar@{=>}[r]
			\ar@{=>}[d]
			&
			{\text{\upshape(iii)}}
			\ar@{=>}[d]
			&
			{\text{\upshape(iv)}}
			\ar@{=>}[d]
			&
			\\
			{\text{\upshape(v)}}
			\ar@{=>}[r]
			&
			{\text{\upshape(vi)}}
			\ar@{=>}[r]
			&
			{\text{\upshape(vii)}}
			\ar@{=>}[r]
			&
			{\text{\upshape(viii)}}
			\ar@{<=>}[r]
			&
			{\text{\upshape(ix)}}
		} 
	} 
\end{proposition}
\begin{proof}
	``(i) $\Rightarrow$ (ii)'' This is obvious.

	``(i) $\Rightarrow$ (iv)'' This is obvious.
		
	``(ii) $\Rightarrow$ (iii)'' This was proved in \cite[Proposition~2.6]{GlueckWolffLB}.
		
	``(ii) $\Rightarrow$ (vi)'' The proof of this implication has already been sketched in \cite[Remark~6.2]{GlueckWolffLB}; we give a few more details here: 
	
	Let $P: X'' \to X''$ be the band projection with range $X$ and let $(x_j)$ be an increasing and norm-bounded net in $X_+$. Then $(x_j)$ converges to a vector $x'' \in X''_+$ with respect to the weak${}^*$-topology. We have $Px'' \le x''$. On the other hand, $Px'' \ge Px_j = x_j$ for each index $j$, which implies that $Px'' \ge x''$. We have thus shown that $Px'' = x''$, i.e., $x := x''$ is an element of $X$.
	
	The increasing net $(x_j)$ converges weakly to $x$, so it follows from \cite[Theorem~V.4.3]{Schaefer1999} that $(x_j)$ actually converges in norm to $x$.
	
	``(iii) $\Rightarrow$ (vii)'' Assertion~(iii) implies that every order interval in $X$ is bounded; hence, the cone $X_+$ is normal. Now, let $(x_j)$ is an increasing net in $X_+$ that is bounded above. Then $(x_j)$ is contained in an order interval. Hence, $(x_j)$ is weakly convergent and therefore also norm convergent according to \cite[Theorem~V.4.3]{Schaefer1999}.
	
	``(iv) $\Rightarrow$ (viii)'' Let $A \subseteq X$ be a non-empty upwards directed set that is bounded above. Then the increasing net $(a)_{a \in A}$ is weak${}^*$-convergent to an element $x \in X$, and one readily checks that $x$ is the supremum of $A$.
	
	``(v) $\Rightarrow$ (vi)'' Let $(x_j) \subseteq X_+$ be an increasing norm bounded net. We show that this net is Cauchy and thus norm convergent. To this end, set $\alpha := \sup_j \|x_j\| \in [0,\infty)$ and let $\varepsilon > 0$. Choose $j_0$ such that $\|x_{j_0}\| \ge \alpha - \varepsilon$. For all indices $j \ge j_0$ we then obtain
	\begin{align*}
		\alpha \ge \|x_j\| = \|x_{j_0}\| + \|x_j - x_{j_0}\| \ge \alpha - \varepsilon + \|x_j - x_{j_0}\|
	\end{align*}
	so $\|x_j - x_{j_0}\| \le \varepsilon$. This proves that $(x_j)$ is indeed Cauchy.
	
	``(vi) $\Rightarrow$ (vii)'' As every increasing norm-bounded sequence in $X_+$ is norm convergent, it follows that $X_+$ is normal; see for instance \cite[Theorem~2.45]{Aliprantis2007}. Hence, every increasing net in $X_+$ which is bounded above is also norm bounded and thus norm convergent according to~(vi).
	
	``(vii) $\Rightarrow$ (viii)'' Let $D \subseteq X$ be a non-empty upwards directed set which is bounded above by a vector $u \in X$. Choose $b \in X_+$ such that $b+D$ intersects the positive cone $X_+$. Then $\tilde D := X_+ \cap (b+D)$ is an upwards directed set, too, and $\tilde D$ is bounded above by $u+b$. Hence, the increasing net $(x)_{x \in \tilde D}$ converges to a vector $y \in X$. Clearly, $y$ is the supremum of $\tilde D$, and thus it is also the supremum of $b+D$ (here we used again that $b+D$ is directed). Therefore, $D$ has the supremum $y-b$.

	``(viii) $\Rightarrow$ (ix)'' This is obvious.
	 
	``(ix) $\Rightarrow$ (viii)'' A general result in the theory of ordered sets says that, if every non-empty totally ordered subset of a partially ordered set $Z$ has a supremum in $Z$, then every non-empty upwards directed subset of $Z$ has a supremum in $Z$, too; see for instance \cite[Proposition~1.5.9]{Cohn1981}. We can apply this to the partially set
	\begin{align*}
		Z := X \cup \{\infty\},
	\end{align*}
	where we define $\infty$ as an object that is larger than each element of $X$: 
	
	It follows from~(ix) that every non-empty totally ordered subset of $X \cup \{\infty\}$ has a supremum in $X \cup \{\infty\}$. So if $D \subseteq X$ is non-empty, directed and bounded above by an element $u \in X$, then we first conclude that $D$ has a supremum $s$ in $X \cup \{\infty\}$; since $D$ is bounded above by $u$, it follows that $s \le u$, so in particular, $s \in X$. Now one can immediately check that $s$ is also the supremum of $D$ within $X$.
\end{proof}

As a consequence of Theorem~\ref{thm:weakly-pervasive-monotonically-complete-spaces-are-lattices} we observe that if the pre-Riesz space $X$ is a weakly pervasive ordered Banach space and satisfies at least one of the assertions~(i)--(ix) in Proposition~\ref{prop:regularity-properties-of-ordered-banach-spaces}, then $X$ is actually a vector lattice. Since every finite-dimensional Banach space is reflexive and every generating closed cone in such a space is normal, we obtain in particular the following corollary.

\begin{corollary} \label{cor:finite-dimensional-weakly-pervasive-spaces}
	Let $X$ be finite dimensional and assume that $X_+$ is closed. If $X$ is weakly pervasive, then $X$ is a vector lattice (and thus isomorphic to $\mathbb{R}^n$ with the standard order).
\end{corollary}

We note once again that, for the special case where $X$ is pervasive, Corollary~\ref{cor:finite-dimensional-weakly-pervasive-spaces} has been recently proved in \cite[Theorem~39]{KalauchPreprint4}. Let us remark a few further consequences of Theorem~\ref{thm:weakly-pervasive-monotonically-complete-spaces-are-lattices} in conjuction with Proposition~\ref{prop:regularity-properties-of-ordered-banach-spaces}.

\begin{remarks}
	\begin{enumerate}[(a)]
		\item Recall that several examples of ordered Banach spaces that are pervasive (and hence weakly pervasive) are listed in Examples~\ref{ex:list-of-weakly-pervasive-spaces}. Theorem~\ref{thm:weakly-pervasive-monotonically-complete-spaces-are-lattices} and Proposition~\ref{prop:regularity-properties-of-ordered-banach-spaces} show that such examples have to satisfy many restrictions if we do not want to end up in the category of vector lattices.
		
		\item If the pre-Riesz space $X$ is an ordered Banach space with normal cone, then the dual space $X'$ also has generating cone and is thus a pre-Riesz space. Proposition~\ref{prop:regularity-properties-of-ordered-banach-spaces} shows that every non-empty totally ordered set in $X'$ that is bounded above has a supremum. Hence, if $X'$ is weakly pervasive, it follows from Theorem~\ref{thm:weakly-pervasive-monotonically-complete-spaces-are-lattices} that $X'$ is in fact a vector lattice, and thus we conclude in turn that $X$ has the Riesz decomposition property \cite[Theorem~2.47]{Aliprantis2007}.
		
		Hence, the dual space $X'$ of an ordered Banach space $X$ with normal (and generating) cone cannot be weakly pervasive unless $X$ itself has the Riesz decomposition property. This suggests that the property ``weakly pervasive'' is not particularly well-behaved with respect to duality (at least not in the category of ordered Banach spaces).
	\end{enumerate}
\end{remarks}

\subsection{Criteria in terms of other concepts of disjointness} \label{subsection:criteria-in-terms-of-other-concepts-of-disjointness}

Recall that weakly pervasive spaces are precisely those pre-Riesz spaces in which any two D-disjoint elements of the positive cone are automatically disjoint. In this context is is interesting to observe that, in general pre-Riesz spaces, there exists an intermediate concept between disjointness and D-disjointness; this is the content of the following proposition.

\begin{proposition} \label{prop:intermediate-between-disjoint-and-d-disjoint}
	For all $x,y \in X_+$ we have the following implications:
	\begin{align*}
		& x\text{ and } y \text{ are disjoint} \\
		& \qquad \qquad \Rightarrow \quad [-x,x] \cap [-y,y] = \{0\} \\
		& \qquad \qquad \qquad \qquad \Rightarrow \quad x\text{ and } y \text{ are D-disjoint.}
	\end{align*}
\end{proposition}
\begin{proof}
	Assume first that $x$ and $y$ are disjoint. If $f \in [-x,x] \cap [-y,y]$, then $f$ is a lower bound of $x$ and $y$, so $f \le 0$. Moreover, $-f$ is also a lower bound of $x$ and $y$, so $-f \le 0$. Hence, $f= 0$. The second implication is obvious.
\end{proof}

We will see in Example~\ref{ex:disjointness-implication-cannot-be-reversed} below that none of the two implications in Proposition~\ref{prop:intermediate-between-disjoint-and-d-disjoint} can be reversed in general pre-Riesz spaces. Before we give this example, we need a small auxiliary result.

We recall from \cite[Definition~27]{KalauchPreprint4} that an element $a \in X_+ \setminus \{0\}$ is called an \emph{atom} in $X$ if every vector $x \in [0,a]$ is a multiple of $a$; equivalently, the order interval $[0,a]$ equals the line segment $\{\lambda a: \; \lambda \in [0,1]\}$.

\begin{lemma} \label{lem:symmetric-line-segments-for-atom}
	Let $a$ be an atom in $X$. Then the order interval $[-a,a]$ equals the line segment $\{\lambda a: \; \lambda \in [-1,1]\}$.
\end{lemma}
\begin{proof}
	Each $x \in [-a,a]$ can be written as
	\begin{align*}
		x = \frac{a+x}{2} - \frac{a-x}{2},
	\end{align*}
	where both $\frac{a+x}{2}$ and $\frac{a-x}{2}$ are elements of $[0,a]$; hence, we have $[-a,a] = [0,a] - [0,a]$ (for this, we did not use that $a$ is an atom). Since $[0,a]$ is the line segment $\{\lambda a: \; \lambda \in [0,1]\}$, this implies the assertion.
\end{proof}

\begin{example} \label{ex:disjointness-implication-cannot-be-reversed}
	Let $X = \bbR^3$, let $X_+$ denote the four ray cone from Example~\ref{ex:four-ray-cone-bands}, and let $v_1,\dots,v_4 \in X$ be the vectors from that example.
	\begin{enumerate}[(a)]
		\item Let $w = v_1+v_2$ and $\tilde w = v_3 + v_4$. Then $w$ and $\tilde w$ are $D$-disjoint, but the set $[-w,w] \cap [-\tilde w, \tilde w]$ is non-zero since it contains the vector $(1, -1, 0)^T$.
				
		\item The order interval $[-v_1,v_1]$ is precisely the line segment $\{\lambda v_1: \; \lambda \in [-1,1]\}$; this follows from Lemma~\ref{lem:symmetric-line-segments-for-atom} since $v_1$ is an atom in $X$ (which in turn follows from \cite[Proposition~28]{KalauchPreprint4}). Similarly, the order interval $[-v_2,v_2]$ is the line segment $\{\lambda v_2: \; \lambda \in [-1,1]\}$.
		
		We thus conclude that $[-v_1,v_1] \cap [-v_2,v_2] = \{0\}$. Yet, we have seen in Example~\ref{ex:non-disjoint-bands-with-trivial-intersection} that $v_1$ and $v_2$ are not disjoint.
	\end{enumerate}
\end{example}

We now consider pre-Riesz spaces in which, for all $x,y \in X_+$, the property $[-x,x] \cap [-y,y] = \{0\}$ implies that $x \perp y$. This property of a pre-Riesz space is (at least formally) weaker than being weakly pervasive. In finite dimensions, though, this property still suffices to conclude that a pre-Riesz space with closed cone is a vector lattice; we prove this in the following theorem.

\begin{theorem} \label{thm:weaker-version-of-weakly-pervasive-implies-vector-lattice}
	Let $X$ be finite dimensional and assume that $X_+$ is closed. Suppose that all vectors $x,y \in X_+$ that satisfy $[-x,x] \cap [-y, y] = \{0\}$ are disjoint. Then $X$ is a vector lattice.
\end{theorem}

For the proof we need the notion of an \emph{extreme ray}. Let $a \in X_+ \setminus \{0\}$. If the half ray $\{\lambda a: \; \lambda \in [0,\infty)\}$ is a face of $X_+$, then we call this half ray an extreme ray of $X_+$. We note that $\{\lambda a: \; \lambda \in [0,\infty)\}$ is an extreme ray of $X_+$ if and only if $a$ is an atom in $X$ (\cite[Proposition~28]{KalauchPreprint4}). If $X$ is finite dimensional and non-zero and $X_+$ is closed, then $X_+$ is the convex hull of its extreme rays.

\begin{proof}[Proof of Theorem~\ref{thm:weaker-version-of-weakly-pervasive-implies-vector-lattice}]
	Set $n := \dim X$; we may assume that $n \ge 1$. Let $E$ denote the set of all extreme rays of $X_+$ and for each $R \in E$, choose a non-zero vector $x_R \in R$. Then the set $\{x_R: \, R \in E\}$ spans $X$, so $E$ has at least $n$ elements.
	
	On the other hand, each point $x_R$ is an atom in $X$, so for any two distinct rays $R,S \in E$ we have $[-x_R,x_R] \cap [-x_S, x_S] = \{0\}$ according to Lemma~\ref{lem:symmetric-line-segments-for-atom}. Thus, it follows from the assumption that $x_S \perp x_R$ for any two distinct rays $R,S \in E$. Hence, we conclude from Proposition~\ref{prop:disjoint-vectors-are-linearly-independent} that the family of vectors $(x_R)_{R \in E}$ is linearly independent. Hence, $E$ has exactly $n$ elements. This proves that the positive cone $X_+$ is generated by exactly $n$ extreme rays, so $X$ is a vector lattice.
\end{proof}

\subsection*{Acknowledgements}

It is my pleasure to thank Anke Kalauch, Helena Malinowski and Onno van Gaans for various suggestions and discussions which helped me to considerably improve the paper.

I am also indebted to Andreas Blass for an answer on MathOverflow \cite{Blass2020} that was very helpful for the proof of the implication ``(ix) $\Rightarrow$ (viii)'' of Proposition~\ref{prop:regularity-properties-of-ordered-banach-spaces}.

This paper was originally motivated by the question how many bands and projection bands can exist in a finite dimensional pre-Riesz space -- a question I first became aware of during a plenary talk of Anke Kalauch at the \emph{Positivity X} conference that took place in July 2019 in Pretoria, South Africa. I am indebted to the organisers of the conference for financially supporting my participation.

\bibliographystyle{plain}
\bibliography{literature}

\end{document}